\documentclass[11pt]{article}

\usepackage[margin=2cm]{geometry}

\usepackage{amsmath}
\usepackage{amssymb}
\usepackage{amsthm}
\usepackage{amsfonts}
\usepackage{enumerate}
\usepackage{graphicx}
\usepackage{color}
\usepackage{verbatim}

\usepackage{multicol}
\usepackage{mathrsfs}
\usepackage{mathtools}

\usepackage{ytableau}

\usepackage{subfig}
\usepackage{tikz}

\usetikzlibrary{arrows,shapes,positioning,decorations.markings}

\tikzset{->-/.style={decoration={
  markings,
  mark=at position .5 with {\arrow{>}}},postaction={decorate}}}

\usepackage{array}

\usepackage{xcolor,colortbl}

\usepackage{soul}

\newtheorem{theorem}{Theorem}[section]

\newtheorem{corollary}[theorem]{Corollary}

\newtheorem{proposition}[theorem]{Proposition}
\newtheorem*{introthm}{Theorem}

\def\R{\mathbb{R}}
\def\Z{\mathbb{Z}}

\begin{document}

\title{Annular Non-Crossing Matchings}
\author{Paul Drube \qquad Puttipong Pongtanapaisan\\
\small Department of Mathematics and Statistics\\
\small Valparaiso University\\
\small \tt \{paul.drube, puttipong.pongtanapaisan\}@valpo.edu 
}

\maketitle

\begin{abstract}
It is well known that the number of distinct non-crossing matchings of $n$ half-circles in the half-plane with endpoints on the x-axis equals the $n^{th}$ Catalan number $C_n$.  This paper generalizes that notion of linear non-crossing matchings, as well as the circular non-crossings matchings of Goldbach and Tijdeman, to non-crossings matchings of $n$ line segments embedded within an annulus.  We prove that the number of such matchings $\vert Ann(n,m) \vert$ with $n$ exterior endpoints and $m$ interior endpoints correspond to an entirely new, one-parameter generalization of the Catalan numbers with $C_n = \vert Ann(1,m) \vert$.  We also develop bijections between specific classes of annular non-crossing matchings and other combinatorial objects such as binary combinatorial necklaces and planar graphs.  Finally, we use Burnside's Lemma to obtain an explicit formula for $\vert Ann(n,m) \vert$ for all $n,m \geq 0$.
\end{abstract}

\section{Introduction}
\label{sec: background}

The Catalan numbers are arguably the most studied sequence of positive integers in mathematics.  Among their seemingly countless combinatorial interpretations is an identification of the $n^{th}$ Catalan number $C_n = \frac{1}{n+1} \binom{2n}{n}$ with non-crossing matchings of $2n$ points along the x-axis via $n$ half-circles in the upper half-plane.  In an effort to avoid confusion, we will sometimes refer to such arrangements as linear non-crossing matchings of order $n$.  The $n^{th}$ Catalan number is also known to equal the number of ordered rooted trees with $n$ non-root vertices.  One bijection between these two interpretations is shown in Figure \ref{fig: matchings to trees, half-plane}.  That map involves placing a vertex in each region of the complement of a matching, with the ``external" region receiving the root vertex, and then adding an edge if two regions are separated by a half-circle.  For an extended treatment of the many different interpretations of Catalan numbers and the bijections between them, see \cite{Stanley}.

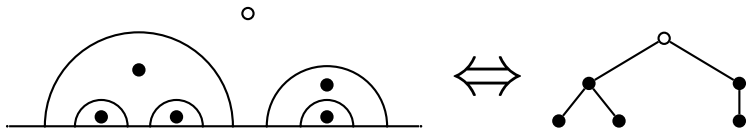
\begin{figure}[h!]
\centering
\begin{tikzpicture}
[scale=.5,auto=left,every node/.style={circle,fill=black,inner sep=0pt}]
	\node[inner sep=0pt] (left) at (0,0) {};
	\node (1) at (1,0) {};
	\node (2) at (1.8,0) {};
	\node (3) at (3.2,0) {};
	\node (4) at (3.8,0) {};
	\node (5) at (5.2,0) {};
	\node (6) at (6,0) {};
	\node (7) at (6.9,0) {};
	\node (8) at (7.8,0) {};
	\node (9) at (9.2,0) {};
	\node (10) at (10.1,0) {};
	\node[inner sep=0pt] (right) at (11,0) {};
	\draw[thick] (left) to (right);
	\draw[thick] (3.2,0) arc (0:180:.7);
	\draw[thick] (5.2,0) arc (0:180:.7);
	\draw[thick] (6,0) arc (0:180:2.5);
	\draw[thick] (10.1,0) arc (0:180:1.6);
	\draw[thick] (9.2,0) arc (0:180:.7);
	\node[fill=white,draw,inner sep=1.5pt,thick] (n1) at (6.4,3) {};
	\node[fill=black,draw,inner sep=1.5pt,thick] (n2) at (3.5,1.5) {};
	\node[fill=black,draw,inner sep=1.5pt,thick] (n3) at (8.5,1.1) {};
	\node[fill=black,draw,inner sep=1.5pt,thick] (n4) at (2.5,0.25) {};
	\node[fill=black,draw,inner sep=1.5pt,thick] (n5) at (4.5,0.25) {};
	\node[fill=black,draw,inner sep=1.5pt,thick] (n6) at (8.5,0.25) {};
\end{tikzpicture}
\hspace{0.05in}
\scalebox{2.5}{\raisebox{5pt}{$\Leftrightarrow$}}
\hspace{0.05in}
\begin{tikzpicture}
[scale=0.4,auto=left,every node/.style={circle,fill=black,inner sep=2.5pt}]
	\node[fill=white,draw,inner sep=1.5pt,thick] (n1) at (6,3) {};
	\node[fill=black,draw,inner sep=1.5pt,thick] (n2) at (3.5,1.5) {};
	\node[fill=black,draw,inner sep=1.5pt,thick] (n3) at (8.5,1.5) {};
	\node[fill=black,draw,inner sep=1.5pt,thick] (n4) at (2.5,0.25) {};
	\node[fill=black,draw,inner sep=1.5pt,thick] (n5) at (4.5,0.25) {};
	\node[fill=black,draw,inner sep=1.5pt,thick] (n6) at (8.5,0.25) {};
	\draw[thick] (n1) to (n2);
	\draw[thick] (n1) to (n3);
	\draw[thick] (n2) to (n4);
\draw[thick] (n2) to (n5);
\draw[thick] (n3) to (n6);
\end{tikzpicture}
\caption{The bijection between linear non-crossing matchings and ordered rooted trees.}
\label{fig: matchings to trees, half-plane}
\end{figure}

Non-crossing matchings admit many interesting generalizations if one restricts curves to a subset of $\R^2$ that is not the upper half-plane.  One such modification is what we refer to as circular non-crossing matchings.  In a circular non-crossing matching of order $n$, $2n$ distinct points on the unit circle are connected by a set of $n$ non-intersecting smooth curves within the unit circle.  Circular non-crossing matchings are considered equivalent if they differ by isotopies within the unit circle (including isotopies that ``slide" endpoints), as long as those isotopies do not involve in curves or endpoints intersecting.  In particular, circular matchings related by rotation about the center of the unit circle are equivalent.  However, matchings that can only be related via reflections are considered distinct.  We henceforth refer to the number of circular non-crossing matchings of order $n$, modulo these relations, by $\widetilde{C}_n$.

One in-depth study of circular non-crossing matchings was undertaken by Goldbach and Tijdeman in \cite{GT}.  In that work, an involved application of Burnside's Lemma showed that:

\begin{equation}
\label{eq: Goldbach-Tijdeman enumeration}
\widetilde{C}_n \ = \ \frac{1}{2n} \left( \sum_{d | n} \phi(n/d) \binom{2d}{d} \right) - \frac{1}{2} C_n + \frac{1}{2} C_{(n-1)/2}
\end{equation}

\noindent where $\phi(k)$ is Euler's totient function and $C_k$ is taken to be zero when $k$ is not an integer, so that the final term only appears when $n$ is odd.

Although not obvious from Equation \ref{eq: Goldbach-Tijdeman enumeration}, there is also a bijection between circular non-crossing matchings of order $n$ and unrooted planar trees with $n+1$ nodes ($n$ edges).  For an illustration of this bijection, see Figure \ref{fig: matchings to trees, circular}.  Notice that this correspondence identifies $\widetilde{C}_n$ with the $(n-1)^{st}$ entry of A002995 \cite{OEIS}.  See \cite{Harary} for further discussion of the graph-theoretic interpretations of $\widetilde{C}_n$.

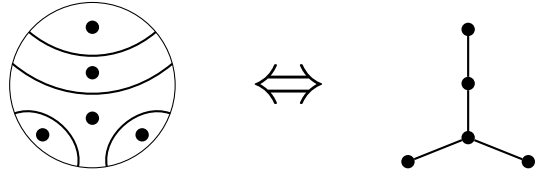
\begin{figure}[h!]
\centering
\begin{tikzpicture}
[scale=1.1,auto=left,every node/.style={circle,inner sep=0pt}]
	\draw (0,0) circle (1cm);
	\node[fill=black,draw,inner sep=1.5pt,thick] (n1) at (90:0.7) {};
	\node[fill=black,draw,inner sep=1.5pt,thick] (n2) at (90:0.15) {};
	\node[fill=black,draw,inner sep=1.5pt,thick] (n3) at (270:0.4) {};
	\node[fill=black,draw,inner sep=1.5pt,thick] (n4) at (225:0.85) {};
	\node[fill=black,draw,inner sep=1.5pt,thick] (n5) at (315:0.85) {};
	\draw[thick, bend right=40] (140:1) to (40:1);
	\draw[thick, bend right=40] (165:1) to (15:1);
	\draw[thick, bend right=60] (340:1) to (280:1);
	\draw[thick, bend right=60] (260:1) to (200:1);
\end{tikzpicture}
\hspace{0.3in}
\scalebox{2.5}{\raisebox{10pt}{$\Leftrightarrow$}}
\hspace{0.3in}
\begin{tikzpicture}
[scale=0.4,auto=left,every node/.style={circle,fill=black,inner sep=2.5pt,thick}]
	\node(n1)[fill=black,draw,inner sep=1.5pt,thick] at (0,-2) {};
	\node(n2)[fill=black,draw,inner sep=1.5pt,thick] at (0,-0.2) {};
	\node(n3)[fill=black,draw,inner sep=1.5pt,thick] at (0,1.6) {};
	\node(n4)[fill=black,draw,inner sep=1.5pt,thick] at (2,-2.8) {};
	\node(n5)[fill=black,draw,inner sep=1.5pt,thick] at (-2,-2.8) {};
	\draw[thick] (n1) to (n2);
	\draw[thick] (n2) to (n3);
	\draw[thick] (n4) to (n1);
	\draw[thick] (n5) to (n1);
\end{tikzpicture}
\caption{The bijection between circular non-crossing matchings and unrooted planar trees.}
\label{fig: matchings to trees, circular}
\end{figure}

In this paper we introduce a new, two-parameter generalization of linear non-crossing matchings in which our smooth curves are embedded within an annulus.  So let $n,m$ be non-negative integers such that $n+m$ is even.  We define an \textbf{annular non-crossing matching of type $(n,m)$} to be a collection of $\frac{n+m}{2}$ non-intersecting smooth curves within the annulus whose endpoints lie at $n+m$ distinct points along the annulus, with $n$ of those endpoints on the exterior boundary of the annulus and $m$ of those endpoints on the interior boundary of the annulus.  Two annular matchings are considered equivalent if they differ by isotopies within the annulus (including isotopies that ``slide" endpoints), as long as those isotopies don't result in curves or endpoints intersecting.  Annular matchings that can only be related by reflections are considered distinct; also disallowed are isotopies where an edge must pass through the hole in the middle of the annulus.  We denote the set of annular non-crossing matchings of type $(n,m)$, modulo these relations, by $Ann(n,m)$.  If we wish to reference the larger collection of all annular non-crossing matchings with $N$ total endpoints, no matter how those endpoints are partitioned between inner and outer boundary components, we write $Ann(N)$.  Thus $Ann(N) = \lbrace M \in Ann(n,m) \ \vert \ n+m = N \rbrace$ and $Ann(N) \neq 0$ if and only if $N$ is even.

In many settings, it will be advantageous to sub-divide annular matchings according to the the number of curves that do not isotope to half-circles on one boundary component.  We define a \textbf{cross-cut} to be a curve in an annular non-crossing matching with one endpoint on the inside of the annulus and one endpoint on the outside of the annulus.  We denote the set of annular non-crossing matchings of type $(n,m)$ with precisely $k$ cross-cuts by $Ann_k(n,m)$, so that $Ann(n,m) = \bigcup_k Ann_k(n,m)$.  See Figure \ref{fig: annular matchings examples} for several quick examples.

\begin{figure}[h!]
\centering
\begin{tikzpicture}
[scale=1.2,auto=left,every node/.style={circle,inner sep=0pt}]
\draw (0,0) circle (.35cm);
\draw (0,0) circle (1cm);
\draw[thick] (165:1) to (165:0.35);
\draw[thick] (15:1) to (15:0.35);
\draw[thick] (50:0.35) arc (-13:198:0.25) ;
\draw[thick, bend right=45] (313:1) to (228:1);
\draw[thick, bend right=60] (290:1) to (250:1);
\end{tikzpicture}
\hspace{.7in}
\begin{tikzpicture}
[scale=1.2,auto=left,every node/.style={circle,inner sep=0pt}]
\draw (0,0) circle (.35cm);
\draw (0,0) circle (1cm);
\draw[thick, bend left=45] (0:1) to (60:1);
\draw[thick, bend left=45] (120:1) to (180:1);
\draw[thick, bend left=45] (240:1) to (300:1);
\end{tikzpicture}
\caption{An element of $Ann_2(6,4)$, and an element of $Ann_0(6,0)$.}
\label{fig: annular matchings examples}
\end{figure}
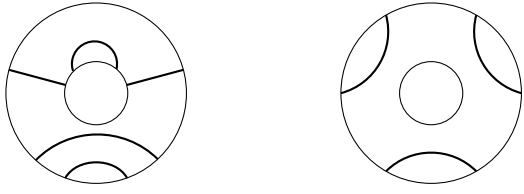

Observe that every annular matching may be isotoped so that all of its cross-cuts appear as straight chords that meet both boundary circles at a right angle. Also notice that $\vert Ann_k(n,m) \vert = 0$ unless $n-k$ and $m-k$ are both even.  When working with an element of $Ann_k(n,m)$, we will sometimes refer to the $(n-k)/2$ curves with both endpoints on the outside of the annulus as ``external half-circles", and to the $(m-k)/2$ curves with both endpoints on the inside of the annulus as ``internal half-circles".

\subsection{Outline of Results}
\label{subsec: outline of results}

The primary goal of this paper is to enumerate $\vert Ann(n,m) \vert$ for arbitrary non-negative integers $n$ and $m$.  This will be accomplished by enumerating $\vert Ann_k(n,m) \vert$ for arbitrary $n,m,k$ and then summing over $k$.  Section \ref{sec: basic annular results} will begin the process with a series of basic results about annular non-crossing matchings, laying the theoretical framework for our subsequent enumerations.  Along the way we will demonstrate a bijection between elements of $Ann_k(n,m)$ and sets of planar graphs that possess a single $k$-cycle and no cycles of any other size (Proposition \ref{thm: graph interpretation, no crosscuts}, Theorem \ref{thm: graph interpretation, general}).

Section \ref{sec: enumeration} presents our enumerative results.  Subsection \ref{subsec: m=0 enumeration and necklaces} begins with a consideration of the sets $Ann_k(2n+k,k)$, a sub-case that we refer to as ``maximal cross-cut annular matchings".  Burnside's Lemma will be used to prove the following, which later appears as Theorem \ref{thm: maximal cross-cut direct enumeration}:

\begin{introthm}
\label{thm: intro theorem 1}
Let $n$ and $k$ be non-negative integers, not both zero.  Then:
\begin{center}
$\vert Ann_k(2n+k,k)\vert = \displaystyle{\frac{1}{2n+k} \sum_{d | (2n+k,n)} \kern-4pt \phi(d) \binom{(2n+k)/d}{n/d}}$
\end{center}
Where $\phi(d)$ is Euler's totient function and $d$ runs over all common divisors of $2n+k$ and $n$.
\end{introthm}

Subsection \ref{thm: maximal cross-cut direct enumeration} will also exhibit a direct bijection between these maximal cross-cut annular matchings and binary combinatorial necklaces.  If $N_2(n_1,n_2)$ denotes the number of binary combinatorial necklaces with $n_1$ black beads and $n_2$ white beads, then Theorem \ref{thm: annular matchings = necklaces} will prove:

\begin{introthm}
\label{thm: intro theorem 2}
Let $n$ and $k$ be non-negative integers.  Then $\vert Ann_k(2n+k,k) \vert = N_2(n+k,n)$.
\end{introthm}

Our results for $\vert Ann_k(2n+k,k) \vert$ are then used in Subsection \ref{subsec: enumeration general} to enumerate $\vert Ann_k(2n+k,2m+k) \vert$ for the remaining choices of $n$, $m$, and $k$.  In particular, Theorem \ref{thm: general enumeration} will show that:

\begin{introthm}
\label{thm: intro theorem 3}
Let $n$, $m$ , and $k$ be non-negative integers with $m > 0$.  Then:
\begin{enumerate}
\item $\vert Ann_k(2n+k,2m+k) \vert = \vert Ann_0(2n,0) \vert \cdot \vert Ann_0(2m,0) \vert$ \ \ if $k = 0$, and
\item $\vert Ann_k(2n+k,2m+k) \vert = \displaystyle{\frac{k}{(2n+k)(2m+k)} \sum_{d | (2n+k,n,m)} \kern-10pt \phi(d) \binom{(2n+k)/d}{n/d} \binom{(2m+k)/d}{m/d}}$ \ \ if $k > 0$
\end{enumerate}
Where $\phi(d)$ is Euler's totient function and summations run over all common divisors of the given integers.
\end{introthm}

We close the paper with an appendix of tables giving outputs for various $Ann(n,m)$, $Ann_k(n,m)$, and $Ann(N)$, all calculated in Maple using Theorems \ref{thm: maximal cross-cut direct enumeration} and \ref{thm: general enumeration}.

\section{Basic Results About Annular Non-Crossing Matchings}
\label{sec: basic annular results}

In this section we present a series of foundational results about annular non-crossing matchings, some of which will be utilized to prove the more general enumerative results of Section \ref{sec: enumeration}.  We also take the opportunity to draw bijections between annular matchings and various classes of planar graphs, and relate the number of ``zero cross-cut" matchings in $Ann(2n,0)$ to $C_n$ and $\widetilde{C}_n$.

\begin{proposition}
\label{thm: reflection equivalence}
Let $n,m$ be non-negative integers.  Then $\vert Ann_k(n,m) \vert = \vert Ann_k(m,n) \vert$ for all $k \geq 0$.  In particular, $\vert Ann(n,m) \vert = \vert Ann(m,n) \vert$ for all $n,m \geq 0$.
\end{proposition}
\begin{proof}
Begin by isotoping elements of $Ann_k(n,m)$ and $Ann_k(m,n)$ so that all cross-cuts appear as straight chords orthogonal to the boundary circles.  Reflection across the ``core" of the annulus then maps every cross-cut to itself, and defines a bijection between $Ann_k(n,m)$ and $Ann_k(m,n)$ for any $k \geq 0$.
\end{proof}

The primary use of Proposition \ref{thm: reflection equivalence} is that it will allow us to restrict our attention to $Ann(n,m)$ and $Ann_k(n,m)$ such that $n \geq m$.  Yet even within the realm where $n \geq m$, there will be specific ``easy" choices for $n$ and $m$ that will prove to have the most useful combinatorial interpretations.  The first of these special cases is $Ann(2n,0) = Ann_0(2n,0)$, corresponding to the situation where all curves in the matchings are external half-circles.

\begin{proposition}
\label{thm: graph interpretation, no crosscuts}
Let $n$ be any non-negative integer.  Then $\vert Ann(2n,0) \vert$ equals the number of unrooted planar trees with a distinguished vertex and $n$ additional vertices.
\end{proposition}
\begin{proof}
The required bijection is analogous to the constructions of Figures \ref{fig: matchings to trees, half-plane} and \ref{fig: matchings to trees, circular} for linear and circular matchings.  We add one vertex for each region in the matching and connect two vertices with an edge if they are separated by a half-circle:
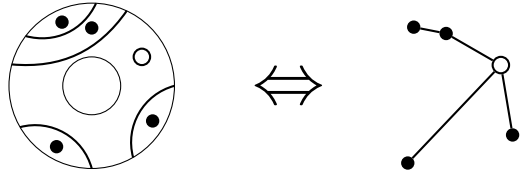
\begin{center}
\begin{tikzpicture}
[scale=1.1,auto=left,every node/.style={circle,inner sep=0pt}]
	\draw (0,0) circle (.35cm);
	\draw (0,0) circle (1cm);
	\draw[thick, bend left=40] (87:1) to (143:1);
	\draw[thick, bend left=30] (65:1) to (165:1);
	\draw[thick, bend left=45] (210:1) to (270:1);
	\draw[thick, bend left=45] (300:1) to (0:1);
	\node[fill=white,draw,inner sep=2.2pt,thick] (n1) at (30:.7) {};
	\node[fill=black,draw,inner sep=1.5pt,thick] (n2) at (115:.85) {};
	\node[fill=black,draw,inner sep=1.5pt,thick] (n3) at (90:.7) {};
	\node[fill=black,draw,inner sep=1.5pt,thick] (n4) at (240:.85) {};
	\node[fill=black,draw,inner sep=1.5pt,thick] (n5) at (330:.85) {};
\end{tikzpicture}
\hspace{0.3in}
\scalebox{2.5}{\raisebox{10pt}{$\Leftrightarrow$}}
\hspace{0.3in}
\begin{tikzpicture}
	[scale=1.2,auto=left,every node/.style={circle,inner sep=0pt}]
	\node[fill=white,draw,inner sep=2.2pt,thick] (n1) at (30:.7) {};
	\node[fill=black,draw,inner sep=1.5pt,thick] (n2) at (115:.85) {};
	\node[fill=black,draw,inner sep=1.5pt,thick] (n3) at (90:.7) {};
	\node[fill=black,draw,inner sep=1.5pt,thick] (n4) at (240:.85) {};
	\node[fill=black,draw,inner sep=1.5pt,thick] (n5) at (330:.85) {};
	\draw[thick] (n1) to (n3);
	\draw[thick] (n3) to (n2);
	\draw[thick] (n1) to (n4);
	\draw[thick] (n1) to (n5);
\end{tikzpicture}
\end{center}
Here the distinguished vertex is placed in the sole region that borders the interior boundary of the annulus.  In placing the edges for our graph, we disregard whether that edge would have passed through the hole in the center of the annulus (we treat the hole as part of the internal region).  As the half-circles in our matchings may be cyclically rotated around the center of the annulus in ``blocks", this gives us the desired notion of equivalence for our planar graphs.
\end{proof}

The class of planar graphs from Proposition \ref{thm: graph interpretation, no crosscuts} is equivalent to rooted planar graphs with $n$ non-root vertices, as long as cyclic reordering of subtrees around the root vertex gives equivalent trees.  This interpretation identifies $\vert Ann(2n,0) \vert$ with the $n^{th}$ entry of A003239 \cite{OEIS}.  See \cite{Harary} and \cite{Stanley2} for further bijections involving $\vert Ann(2n,0) \vert$.

Via the planar graph bijections of Section \ref{sec: background} and Proposition \ref{thm: graph interpretation, no crosscuts}, it is immediate that $\widetilde{C}_n \leq \vert Ann(2n,0) \vert \leq C_n$ for all $n \geq 0$.  One can easily verify that $\widetilde{C}_n = \vert Ann(2n,0) \vert = C_n$ for both $n=0$ and $n=1$.  For $n=2$ we have $\vert Ann(4,0) \vert = C_2 =2$ yet $\widetilde{C}_2 = 1$.  As shown in the following proposition, $n=2$ is the largest value of $n$ for which any of the three quantities are equal:

\begin{proposition}
\label{thm: zero crosscuts inequalities}
For all $n \geq 3$ we have $\widetilde{C}_n \lneq \vert Ann(2n,0) \vert \lneq C_n$. 
\end{proposition}
\begin{proof}
We define a map $\phi$ from the set of all linear non-crossings of order $n$ to $Ann(2n,0)$ by identifying the endpoints of the x-axis and placing the resulting circle as the outer boundary of the annulus.  We then define a map $\psi$ from $Ann(2n,0)$ to the set of all circular non-crossing matchings of order $n$ by ``deleting" the hole in the middle of the annulus.
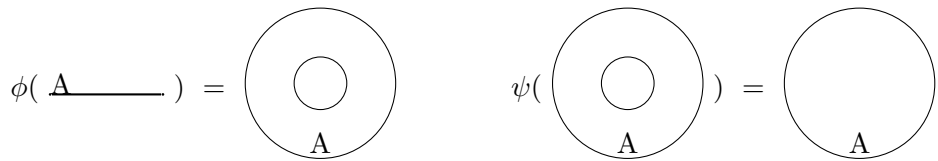
\begin{center}
\raisebox{24pt}{$\phi ($
\begin{tikzpicture}
[scale=1,auto=left,every node/.style={circle,fill=black}]
	\node[inner sep=0pt] (left) at (0,0) {};
	\node[inner sep=0pt] (right) at (1.5,0) {};
	\draw[thick] (left) to (right);
	\put(15,1){$A$}
\end{tikzpicture}
$) \ = \ $}
\begin{tikzpicture}
[scale=1,auto=left,every node/.style={circle,inner sep=0pt}]
	\draw (0,0) circle (1cm);
	\draw (0,0) circle (.35cm);
	\node (A) at (270:.8) {A};
\end{tikzpicture}
\hspace{.5in}
\raisebox{24pt}{$\psi ($}
\begin{tikzpicture}
[scale=1,auto=left,every node/.style={circle,inner sep=0pt}]
	\draw (0,0) circle (1cm);
	\draw (0,0) circle (.35cm);
	\node (A) at (270:.8) {A};
\end{tikzpicture}
\raisebox{24pt}{$) \ = \ $}
\begin{tikzpicture}
[scale=1,auto=left,every node/.style={circle,inner sep=0pt}]
	\draw (0,0) circle (1cm);
	\node (A) at (270:.8) {A};
\end{tikzpicture}
\end{center}
Both $\phi$ and $\psi$ are clearly well-defined and surjective for all $n \geq 0$.  To see that neither map is injective for $n \geq 3$, let $A$ be some non-empty collection of half-circles and notice that:

\begin{center}
\raisebox{18pt}{
\raisebox{4pt}{$\phi ($}
\begin{tikzpicture}
[scale=.5,auto=left,every node/.style={circle,fill=black,inner sep=0pt}]
	\node[inner sep=0pt] (left) at (1,0) {};
	\node[inner sep=0pt] (right) at (6,0) {};
	\draw[thick] (left) to (right);
	\draw[thick] (4.5,0) arc (0:180:.5);
	\draw[thick] (5.5,0) arc (0:180:1.5);
	\put(20,1){$A$}
\end{tikzpicture}
\raisebox{4pt}{$) \ = \ \phi ($}
\begin{tikzpicture}
[scale=.5,auto=left,every node/.style={circle,fill=black,inner sep=0pt}]
	\node[inner sep=0pt] (left) at (1,0) {};
	\node[inner sep=0pt] (right) at (6,0) {};
	\draw[thick] (left) to (right);
	\draw[thick] (4.5,0) arc (0:180:1.5);
	\draw[thick] (3.5,0) arc (0:180:.5);
	\put(70,1){$A$}
\end{tikzpicture}
\raisebox{4pt}{$)$}}
\hspace{.5in}
\raisebox{24pt}{$\psi ($}
\begin{tikzpicture}
[scale=1,auto=left,every node/.style={circle,inner sep=0pt,fill=none}]
	\draw (0,0) circle (.35cm);
	\draw (0,0) circle (1cm);
	\draw[thick,bend left=60] (270:1) to (340:1);
	\node (A) at (235:.8) {A};
\end{tikzpicture}
\raisebox{24pt}{$) \ = \ \psi ($}
\begin{tikzpicture}
[scale=1,auto=left,every node/.style={circle,inner sep=0pt,fill=none}]
	\draw (0,0) circle (.35cm);
	\draw (0,0) circle (1cm);
	\draw[thick,bend left=60] (200:1) to (270:1);
	\node (A) at (235:.8) {A};
\end{tikzpicture}
\raisebox{24pt}{$)$}
\end{center}

\vspace{-.2in}

\end{proof}

With the zero cross-cuts case well-understood, we expand our attention to $Ann_k(n,m)$ with $k \geq 1$.  In what follows we re-index variables to consider $Ann_k(2n+k,2m+k)$, as this alternative notation explicitly references the presence of $n$ external half-circles and $m$ internal half-circles.  The cases that will prove most useful are what we refer to as \textbf{maximal cross-cut annular matchings}.  In maximal cross-cut annular matchings, the only endpoints on the interior boundary belong to cross-cuts.  In our new notation this implies that $m=0$, so that we are dealing with sets of the form $Ann_k(2n+k,k)$.  Notice that the previously considered sets $Ann(2n,0) = Ann_0(2n+0,0)$ qualify as maximal cross-cut annular matchings.

When $k=1$, maximal cross-cut annular matchings are in bijection with the Catalan numbers.  As shown in Figure \ref{fig: 1 crosscut equals Catalan numbers}, a bijection of such annular matchings with linear non-crossing matchings is realized by identifying the outer boundary of the annulus with the real line, in such a way that the outer endpoint of the sole cross-cut corresponds to $\pm \infty$.

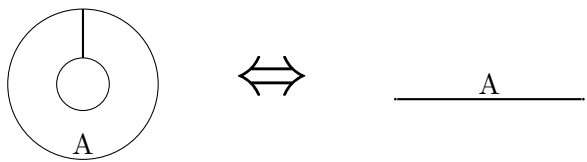
\begin{figure}[h!]
\begin{center}
\begin{tikzpicture}
[scale=1,auto=left,every node/.style={circle,inner sep=0pt,fill=none}]
	\draw (0,0) circle (.35cm);
	\draw (0,0) circle (1cm);
	\draw[thick] (90:1) to (90:0.35);
	\node (A) at (270:.8) {A};
\end{tikzpicture}
\hspace{0.3in}
\scalebox{2.5}{\raisebox{10pt}{$\Leftrightarrow$}}
\hspace{0.3in}
\raisebox{18pt}{
\begin{tikzpicture}
[scale=1,auto=left,every node/.style={circle,fill=black}]
	\node[inner sep=0pt] (left) at (0,0) {};
	\node[inner sep=0pt] (right) at (2.5,0) {};
	\draw[thick] (left) to (right);
	\node[fill=none] (A) at (1.25,.2) {A};
\end{tikzpicture}}
\end{center}

\vspace{-.15in}

\caption{The bijection between $Ann_1(2n+1,1)$ and linear non-crossing matchings of order $n$.}
\label{fig: 1 crosscut equals Catalan numbers}
\end{figure}

Maximal cross-cut annular matchings will be directly enumerated in Subsection \ref{subsec: m=0 enumeration and necklaces} for all $k \geq 0$, and those enumerations will be necessary building blocks for the general enumerations of Subsection \ref{subsec: enumeration general}.  Although the full importance of the maximal cross-cut case will not become obvious until those subsections, we pause to prove one useful property shared by maximal cross-cut matchings and annular matchings with zero cross-cuts:

\begin{proposition}
\label{thm: splitting into one-sided product}
Let $n,m,k$ be non-negative integers.  If $m=0$ or $k=0$ then:
\begin{center}
$\vert Ann_k (2n+k,2m+k) \vert = \vert Ann_k(2n+k,k) \vert \ \vert Ann_k(2m+k,k) \vert$
\end{center}
\end{proposition}
\begin{proof}
We define a map $\phi: Ann_k(2n+k,2m+k) \rightarrow Ann_k(2n+k,k) \oplus Ann_k(k,2m+k)$ that deletes internal half-circles in the first coordinate and deletes external half-circles in the second coordinate, as in the example below:

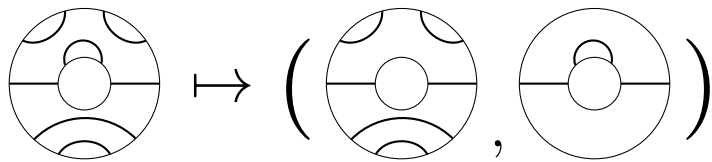
\begin{center}
\begin{tikzpicture}
[scale=1,auto=left,every node/.style={circle,inner sep=0pt}]
	\draw (0,0) circle (.35cm);
	\draw (0,0) circle (1cm);
	\draw[thick] (180:1) to (180:0.35);
	\draw[thick] (0:1) to (0:0.35);
	\draw[thick] (50:0.35) arc (-13:198:0.25) ;
	\draw[thick, bend right=45] (313:1) to (228:1);
	\draw[thick, bend right=60] (290:1) to (250:1);
	\draw[thick, bend left=60] (35:1) to (75:1);
	\draw[thick, bend left=60] (105:1) to (145:1);
\end{tikzpicture}
\hspace{4pt}
\raisebox{22pt}{\scalebox{2.2}{$\mapsto$}}
\hspace{1pt}
\raisebox{17pt}{\scalebox{3.5}{$($}}
\begin{tikzpicture}
[scale=1,auto=left,every node/.style={circle,inner sep=0pt}]
	\draw (0,0) circle (.35cm);
	\draw (0,0) circle (1cm);
	\draw[thick] (180:1) to (180:0.35);
	\draw[thick] (0:1) to (0:0.35);
	\draw[thick, bend right=45] (313:1) to (228:1);
	\draw[thick, bend right=60] (290:1) to (250:1);
	\draw[thick, bend left=60] (35:1) to (75:1);
	\draw[thick, bend left=60] (105:1) to (145:1);
\end{tikzpicture}
\hspace{-10pt}
\raisebox{5pt}{\scalebox{2.0}{$\ , \ $}}
\hspace{-10pt}
\begin{tikzpicture}
[scale=1,auto=left,every node/.style={circle,inner sep=0pt}]
	\draw (0,0) circle (.35cm);
	\draw (0,0) circle (1cm);
	\draw[thick] (180:1) to (180:0.35);
	\draw[thick] (0:1) to (0:0.35);
	\draw[thick] (50:0.35) arc (-13:198:0.25) ;
\end{tikzpicture}
\raisebox{17pt}{\scalebox{3.5}{$)$}}
\end{center}
This map $\phi$ is clearly a bijection whenever $m=0$.  To see that $\phi$ is a bijection when $k=0$, notice that the lack of cross-cuts in the $k=0$ case means that the internal half-circles and external half-circles may be isotoped independently around their respective boundary components and hence ``do not interact".
\end{proof}

It can be shown that the equality of Proposition \ref{thm: splitting into one-sided product} holds precisely when $n=0$ or $m=0$ or $k<2$.  If $k \geq 2$, $n \geq 1$, and $m \geq 1$ all hold, the left side of the expression always proves to be strictly larger than the right side.  However, only the $m = 0$ and $k = 0$ cases are needed in Section \ref{sec: enumeration}, motivating our omission of the more general result.

For our final result of this section, we adapt the planar graph bijection of Proposition \ref{thm: graph interpretation, no crosscuts} to the general case of $Ann_k(2n+k,2m+k)$.  Although the resulting language of Theorem \ref{thm: graph interpretation, general} is arguably rather contrived, it combines with Proposition \ref{thm: graph interpretation, no crosscuts} to give a succinct geometric characterization of any subset of annular matchings that is not $Ann_0(2n,2m)$ with $n,m > 0$.

\begin{theorem}
\label{thm: graph interpretation, general}
Let $n,m,k$ be non-negative integers and let $k \geq 1$.  Then $\vert Ann_k(2n+k,2m+k) \vert$ equals the number of connected planar graphs such that:
\begin{enumerate}
\item The only cycle in the graph is a single $k$-cycle.
\item There are $n$ edges within the cycle.
\item There are $m$ edges outside the cycle.
\end{enumerate}
\end{theorem}
\begin{proof}
The methodology required is extremely similar to what has already been presented in Proposition \ref{thm: graph interpretation, no crosscuts}.  Notice that a collection of $k$ cross-cuts produces a single $k$-cycle, as below:
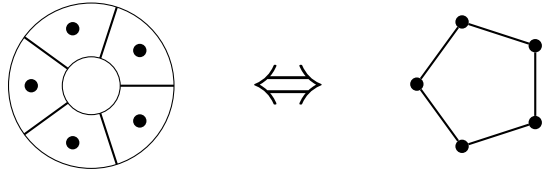
\begin{center}
\begin{tikzpicture}
[scale=1.1,auto=left,every node/.style={circle,inner sep=0pt}]
	\draw (0,0) circle (.35cm);
	\draw (0,0) circle (1cm);
	\draw[thick] (0:1) to (0:.35);
	\draw[thick] (72:1) to (72:.35);
	\draw[thick] (144:1) to (144:.35);
	\draw[thick] (216:1) to (216:.35);
	\draw[thick] (288:1) to (288:.35);
	\node[fill=black,draw,inner sep=1.5pt,thick] (n1) at (36:.725) {};
	\node[fill=black,draw,inner sep=1.5pt,thick] (n2) at (108:.725) {};
	\node[fill=black,draw,inner sep=1.5pt,thick] (n3) at (180:.725) {};
	\node[fill=black,draw,inner sep=1.5pt,thick] (n4) at (252:.725) {};
	\node[fill=black,draw,inner sep=1.5pt,thick] (n5) at (324:.725) {};
\end{tikzpicture}
\hspace{0.3in}
\scalebox{2.5}{\raisebox{10pt}{$\Leftrightarrow$}}
\hspace{0.3in}
\raisebox{6pt}{
\begin{tikzpicture}
	[scale=1.2,auto=left,every node/.style={circle,inner sep=0pt}]
	\node[fill=black,draw,inner sep=1.5pt,thick] (n1) at (36:.725) {};
	\node[fill=black,draw,inner sep=1.5pt,thick] (n2) at (108:.725) {};
	\node[fill=black,draw,inner sep=1.5pt,thick] (n3) at (180:.725) {};
	\node[fill=black,draw,inner sep=1.5pt,thick] (n4) at (252:.725) {};
	\node[fill=black,draw,inner sep=1.5pt,thick] (n5) at (324:.725) {};
	\draw[thick] (n1) to (n2);
	\draw[thick] (n2) to (n3);
	\draw[thick] (n3) to (n4);
	\draw[thick] (n4) to (n5);
	\draw[thick] (n5) to (n1);
\end{tikzpicture}}
\end{center}
Any internal half-circles are then in bijection with edges inside the $k$-cycle, while external half-circles are in bijection with edges outside the $k$-cycle.
\end{proof}

\section{Enumeration of Annular Matchings}
\label{sec: enumeration}

We are now ready for the general enumerative results that form the core of this paper.  Subsection \ref{subsec: m=0 enumeration and necklaces} begins with an enumeration of maximal cross-cut annular matchings, and shows that those matchings are in bijection with certain types of binary combinatorial necklaces.  Subsection \ref{subsec: enumeration general} then collects all of our results to give an explicit formula for general $\vert Ann_k(2n+k,2m+k) \vert$, thus allowing for the direct calculation of $\vert Ann(a,b) \vert$ and $\vert Ann(N) \vert$.

\subsection{Enumeration of $Ann_k(2n+k,k)$ and Combinatorial Necklaces}
\label{subsec: m=0 enumeration and necklaces}

Burnside's Lemma has already been mentioned as the method used in \cite{GT} to calculate the number of circular non-crossing matchings.  Given that our annular non-crossing matchings possess a similar notion of rotational equivalence, it comes as little surprise that the lemma may also be applied to the enumeration of distinct matchings in the annulus.  Recall that Burnside's Lemma (also known as the Cauchy-Frobenius Lemma) applies to any situation where a finite group $G$ acts upon a set $A$.  It asserts that the number of orbits $\vert A/G \vert$ with respect to the action equals the average size of the sets $A^g = \lbrace a \in A \ \vert \ ga = a \rbrace$ when ranging over all $g \in G$: that $\vert A/G \vert = \frac{1}{\vert G \vert} \sum_{g \in G} \vert A^g \vert$.

So fix $n,k \geq 0$, and let $A$ equal the set of all non-crossing matchings in the annulus (prior to any notion of equivalence) with $2n+k$ endpoints located at $\frac{2\pi}{2n+k}$ radian intervals about the exterior boundary and precisely $k$ straight cross-cuts that meet both boundary components orthogonally.  We may define a (left) action of $G = \Z_{2n+k}$ on $A$ whereby $g \cdot a$ is counter-clockwise rotation of $a$ by $\frac{2\pi g}{2n+k}$ radians.  Then $G/A = Ann_k(2n+k,k)$, with distinct orbits in $G/A$ corresponding to matchings that are equivalent via rotation.  This sets up the following application of Burnside's Lemma:

\begin{theorem}
\label{thm: maximal cross-cut direct enumeration}
Let $n$ and $k$ be non-negative integers, not both zero.  Then:
\begin{center}
$\vert Ann_k(2n+k,k)\vert = \displaystyle{\frac{1}{2n+k} \sum_{d | (2n+k,n)} \kern-4pt \phi(d) \binom{(2n+k)/d}{n/d}}$
\end{center}
Where $\phi(d)$ is Euler's totient function and the sum runs over all common divisors $d$ of $2n+k$ and $n$.
\end{theorem}
\begin{proof}
Using the aforementioned group action, by Burnside's Lemma we merely need to show that $ \sum_{g \in \Z_{2n+k}} \vert A^g \vert = \sum_{d | (2n+k,n)} \phi(d) \binom{(2n+k)/d}{n/d}$.  So take $g \in \Z_{2n+k}$, and assume that $g$ has order $d$ in $\Z_{2n+k}$  Notice that Lagrange's Theorem guarantees $d \mid (2n+k)$, although it may or may not be true that $d \mid n$.

The elements of $A^g$ are those matchings that can be radially divided into $d$ identical sub-matchings.  Each of these sub-arrangements features $(2n+k)/d$ endpoints on the outer boundary of the annulus, $k/d$ of which are the outer endpoints of cross-cuts and $n/d$ of which are left-endpoints of exterior half-circles.  Now if $d \nmid n$, the left-endpoints cannot be sub-divided in this way and we may conclude that $\vert A^g \vert = 0$.  However, $d \mid n$ and $d \mid (2n+k,n)$ together guarantee that $k \mid n$, making $\vert A^g \vert \neq 0$ a possibility.  If $d \mid n$, notice that every possible sub-arrangement may be uniquely identified by specifying which of the $(2n+k)/d$ endpoints correspond to left (clockwise) endpoints of exterior half-circles.  This bijection, which closely resembles the upcoming construction in the proof of Theorem \ref{thm: annular matchings = necklaces}, involves recursively connecting each specified endpoint to the nearest available non-specified endpoint on its right and then associating the $k$ unused endpoints with cross-cuts.  See Figure \ref{fig: sub-matching example} for an example.  It follows that $\vert A^g \vert = \binom{(2n+k)/d}{n/d}$ whenever $\vert g \vert = d$ and $d \mid n$.

If $q \mid N$, basic number theory ensures that there are precisely $\phi(q)$ elements $i \in \Z_N$ with greatest common divisor $(i,N) = N/q$.  As the order of any element in $\Z_N$ is $N/(i,N)$, there exist precisely $\phi(q)$ elements $i \in \Z_N$ with order $\vert i \vert = q$.  Letting $N = 2n+k$, this ensures that there are precisely $\phi(d)$ elements $g \in \Z_{2n+k}$ such that $\vert A^g \vert = \binom{(2n+k)/d}{n/d}$, thus deriving the summation of the theorem.
\end{proof}

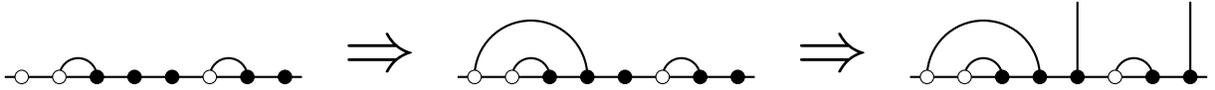
\begin{figure}[h!]
\centering
\begin{tikzpicture}
[scale=.5,auto=left,every node/.style={circle,inner sep=0pt}]
	\node[inner sep=0pt] (left) at (.5,0) {};
	\node[inner sep=0pt] (right) at (8.5,0) {};
	\draw[thick] (left) to (right);
	\draw[thick] (3,0) arc (0:180:.5);
	\draw[thick] (7,0) arc (0:180:.5);
	\node[fill=white,draw,inner sep=1.8pt] (1) at (1,0) {};
	\node[fill=white,draw,inner sep=1.8pt] (2) at (2,0) {};
	\node[fill=black,draw,inner sep=1.8pt] (3) at (3,0) {};
	\node[fill=black,draw,inner sep=1.8pt] (4) at (4,0) {};
	\node[fill=black,draw,inner sep=1.8pt] (5) at (5,0) {};
	\node[fill=white,draw,inner sep=1.8pt] (6) at (6,0) {};
	\node[fill=black,draw,inner sep=1.8pt] (7) at (7,0) {};
	\node[fill=black,draw,inner sep=1.8pt] (8) at (8,0) {};
\end{tikzpicture}
\hspace{0.1in}
\scalebox{2.5}{\raisebox{2pt}{$\Rightarrow$}}
\hspace{0.1in}
\begin{tikzpicture}
[scale=.5,auto=left,every node/.style={circle,inner sep=0pt}]
	\node[inner sep=0pt] (left) at (.5,0) {};
	\node[inner sep=0pt] (right) at (8.5,0) {};
	\draw[thick] (left) to (right);
	\draw[thick] (4,0) arc (0:180:1.5);
	\draw[thick] (3,0) arc (0:180:.5);
	\draw[thick] (7,0) arc (0:180:.5);
	\node[draw,fill=white,inner sep=1.8pt] (1) at (1,0) {};
	\node[fill=white,draw,inner sep=1.8pt] (2) at (2,0) {};
	\node[fill=black,draw,inner sep=1.8pt] (3) at (3,0) {};
	\node[fill=black,draw,inner sep=1.8pt] (4) at (4,0) {};
	\node[fill=black,draw,inner sep=1.8pt] (5) at (5,0) {};
	\node[fill=white,draw,inner sep=1.8pt] (6) at (6,0) {};
	\node[fill=black,draw,inner sep=1.8pt] (7) at (7,0) {};
	\node[fill=black,draw,inner sep=1.8pt] (8) at (8,0) {};
\end{tikzpicture}
\hspace{0.1in}
\scalebox{2.5}{\raisebox{2pt}{$\Rightarrow$}}
\hspace{0.1in}
\begin{tikzpicture}
[scale=.5,auto=left,every node/.style={circle,inner sep=0pt}]
	\node[inner sep=0pt] (left) at (.5,0) {};
	\node[inner sep=0pt] (right) at (8.5,0) {};
	\draw[thick] (left) to (right);
	\draw[thick] (4,0) arc (0:180:1.5);
	\draw[thick] (3,0) arc (0:180:.5);
	\draw[thick] (7,0) arc (0:180:.5);
	\draw[thick] (5,0) to (5,2);
	\draw[thick] (8,0) to (8,2);
	\node[fill=white,draw,inner sep=1.8pt] (1) at (1,0) {};
	\node[fill=white,draw,inner sep=1.8pt] (2) at (2,0) {};
	\node[fill=black,draw,inner sep=1.8pt] (3) at (3,0) {};
	\node[fill=black,draw,inner sep=1.8pt] (4) at (4,0) {};
	\node[fill=black,draw,inner sep=1.8pt] (5) at (5,0) {};
	\node[fill=white,draw,inner sep=1.8pt] (6) at (6,0) {};
	\node[fill=black,draw,inner sep=1.8pt] (7) at (7,0) {};
	\node[fill=black,draw,inner sep=1.8pt] (8) at (8,0) {};
\end{tikzpicture}

\caption{For every choice of $n$ left-endpoints (white circles) there is a unique annular sub-matching with $n$ half-circles and $k$ cross-cuts.  Here the relevant piece of the outer boundary is drawn as the real line.}
\label{fig: sub-matching example}
\end{figure}

Table \ref{tab: Annk(2n+k,k)} of Appendix \ref{sec: appendix tables} exhibits values of $\vert Ann_k(2n+k,k) \vert$ for $0 \leq n,k \leq 10$, all calculated in Maple via the equation of Theorem \ref{thm: maximal cross-cut direct enumeration}.  An examination of that table places $\vert Ann_k(2n+k,k) \vert$ into direct correspondence with the $T(n+k,n)$ entry of OEIS sequence A241926 \cite{OEIS}.  Using Proposition \ref{thm: reflection equivalence} when necessary to ensure that $2n+k > n$, $\vert Ann_k(2n+k,k) \vert$ may also be identified with the ``circular binomial coefficient" $T(2n+k,n)$ of OEIS sequence A047996 \cite{OEIS}.  Both of those OEIS sequence reveal a bijection between the $Ann_k(2n+k,k)$ and binary combinatorial necklaces, and in fact a summation equivalent to the one of Theorem \ref{thm: maximal cross-cut direct enumeration} has already been shown to equal the number of binary combinatorial necklaces of certain types \cite{Elash}.  Yet before investigating how these results relate to annular non-crossing matchings, we observe that the formula of Theorem \ref{thm: maximal cross-cut direct enumeration} may be significantly simplified when $k$ is prime: 

\begin{corollary}
\label{thm: maximal cross-cut direction enumeration, prime k}
Let $p$ be a prime integer and let $n$ be any non-negative integer.  Then:
\begin{center}
$\vert Ann_p(2n+p,p) \vert = \displaystyle{\frac{1}{2n+p} \binom{2n+p}{n} + \frac{p-1}{p} \ C_{n/p}}$
\end{center}
Where $C_{n/p}$ is the Catalan number, and is taken to be zero when $n/p$ is not an integer.
\end{corollary}
\begin{proof}
It is a straightforward exercise to calculate that the greatest common divisor of $2n+p$ and $n$ is $(2n+p,n) = 1$ when $(p,n)=1$, as well as that $(2n+p,n) = p$ when $(p,n) = p$.  Theorem \ref{thm: maximal cross-cut direct enumeration} then gives:
\begin{itemize}
\item $\vert Ann_p(2n+p,n) \vert = \frac{1}{2n+p} \binom{2n+p}{n}$ \ \ when $p \nmid n$
\item $\vert Ann_p(2n+p,n) \vert = \frac{1}{2n+p} \binom{2n+p}{n} \ + \ \frac{1}{2n+p}(p-1) \binom{(2n+p)/p}{n/p}$ \ \ when $p \vert n$
\end{itemize}
A simplification of the final term in the second case then yields the desired formula.
\end{proof}

Notice that the more straightforward formula for $\vert Ann_k(2n+k,n) \vert$ in Corollary \ref{thm: maximal cross-cut direction enumeration, prime k} simplifies to the sequences A007595 \cite{OEIS} when $p=2$ and A003441 \cite{OEIS} when $p=3$.

Driven by sequences A241926 and A047996, we now work to develop an explicit bijection between the $\vert Ann_k(2n+k,k) \vert$ and binary combinatorial necklaces, providing a new combinatorial identity that supplements the one in \cite{Elash}.  A $k$-ary combinatorial necklace is a circular arrangement of ``beads" of up to $k$ distinguishable varieties (typically referred to as ``colors"), such that rotations of the beads around the circle are considered equivalent.  Necklaces that are related only via (orientation-reversing) reflection are not considered equivalent.  The number of distinct $k$-ary combinatorial necklace with precisely $n$ total beads is denoted $N_k(n)$.  Combinatorial necklaces are well-studied in the literature, and different classes of combinatorial necklaces are the focus of many integer sequences \cite{OEIS}.

In this paper we deal only with binary ($2$-ary) combinatorial necklaces, whose colors we refer to as ``black" and ``white".  Our results require increased specificity in that we need to designate the number of beads of each color, so we denote the number of distinct binary necklaces with precisely $n_1$ black beads and precisely $n_2$ white beads by $N_2(n_1,n_2)$.  Pause to note that some places in the literature refer to such combinatorial necklaces as ``binary necklaces of weight $n_1$".

\begin{theorem}
\label{thm: annular matchings = necklaces}
Let $n$ and $k$ be non-negative integers.  Then $\vert Ann_k(2n+k,k) \vert = N_2(n+k,n)$.
\end{theorem}
\begin{proof}
Denote the set of all combinatorial necklaces with $n_1$ black beads and $n_2$ white beads by $S$.  We define functions $\phi_1: Ann_k(2n+k,k) \rightarrow S$ and $\phi_2: S \rightarrow Ann_k(2n+k,k)$ that are both injective.  For $\phi_1$ we follow the procedure exemplified below, placing white beads at the left-endpoints of exterior half-circles and black beads at right-endpoints of exterior half-circles as well as at the exterior endpoints of cross-cuts:
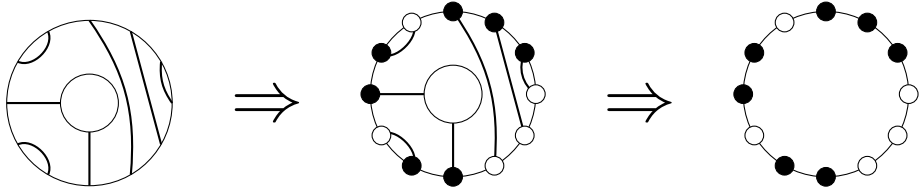
\begin{center}
\begin{tikzpicture}
[scale=1.1,auto=left,every node/.style={circle,inner sep=0pt}]
	\draw (0,0) circle (.35cm);
	\draw (0,0) circle (1cm);
	\draw[thick, bend left=70] (210:1) to (240:1);
	\draw[thick] (270:1) to (270:.35);
	\draw[thick] (180:1) to (180:.35);
	\draw[thick, bend right=70] (150:1) to (120:1);
	\draw[thick, bend right=20] (300:1) to (90:1);
	\draw[thick, bend right=20] (30:1) to (0:1);
	\draw[thick] (60:1) to (330:1);
\end{tikzpicture}
\hspace{0.2in}
\scalebox{2.5}{\raisebox{10pt}{$\Rightarrow$}}
\hspace{0.2in}
\begin{tikzpicture}
[scale=1.1,auto=left,every node/.style={circle,inner sep=0pt}]
	\draw (0,0) circle (.35cm);
	\draw (0,0) circle (1cm);
	\draw[thick, bend left=90] (210:1) to (240:1);
	\draw[thick] (270:1) to (270:.35);
	\draw[thick] (180:1) to (180:.35);
	\draw[thick, bend right=90] (150:1) to (120:1);
	\draw[thick, bend right=20] (300:1) to (90:1);
	\draw[thick, bend right=40] (30:1) to (0:1);
	\draw[thick] (60:1) to (330:1);
	\node[draw,fill=black,inner sep=2.5pt] (1c) at (60:1) {};
	\node[draw,fill=black,inner sep=2.5pt] (2c) at (30:1) {};
	\node[draw,fill=white,inner sep=2.5pt] (3c) at (0:1) {};
	\node[draw,fill=white,inner sep=2.5pt] (4c) at (330:1) {};
	\node[draw,fill=white,inner sep=2.5pt](5c) at (300:1) {};
	\node[draw,fill=black,inner sep=2.5pt](6c) at (270:1) {};
	\node[draw,fill=black,inner sep=2.5pt] (7c) at (240:1) {};
	\node[draw,fill=white,inner sep=2.5pt] (8c) at (210:1) {};
	\node[draw,fill=black,inner sep=2.5pt] (9c) at (180:1) {};
	\node[draw,fill=black,inner sep=2.5pt] (10c) at (150:1) {};
	\node[draw,fill=white,inner sep=2.5pt] (11c) at (120:1) {};
	\node[draw,fill=black,inner sep=2.5pt] (12c) at (90:1) {};
\end{tikzpicture}
\hspace{0.2in}
\scalebox{2.5}{\raisebox{10pt}{$\Rightarrow$}}
\hspace{0.2in}
\begin{tikzpicture}
[scale=1.1,auto=left,every node/.style={circle,inner sep=0pt}]
	\draw (0,0) circle (1cm);
	\node[draw,fill=black,inner sep=2.5pt] (1c) at (60:1) {};
	\node[draw,fill=black,inner sep=2.5pt] (2c) at (30:1) {};
	\node[draw,fill=white,inner sep=2.5pt] (3c) at (0:1) {};
	\node[draw,fill=white,inner sep=2.5pt] (4c) at (330:1) {};
	\node[draw,fill=white,inner sep=2.5pt](5c) at (300:1) {};
	\node[draw,fill=black,inner sep=2.5pt](6c) at (270:1) {};
	\node[draw,fill=black,inner sep=2.5pt] (7c) at (240:1) {};
	\node[draw,fill=white,inner sep=2.5pt] (8c) at (210:1) {};
	\node[draw,fill=black,inner sep=2.5pt] (9c) at (180:1) {};
	\node[draw,fill=black,inner sep=2.5pt] (10c) at (150:1) {};
	\node[draw,fill=white,inner sep=2.5pt] (11c) at (120:1) {};
	\node[draw,fill=black,inner sep=2.5pt] (12c) at (90:1) {};
\end{tikzpicture}
\end{center}

For $\phi_2$ we begin at any point along the combinatorial necklace and proceed counter-clockwise.  Every time we encounter a white bead, we add a half-circle connecting that bead to the first black bead (in the counter-clockwise direction) that is not already the right-endpoint of a half-circle.  Repeat this procedure, traversing the necklace multiple times if necessary, until every white bead is the left-endpoint of a half-circle.  Then add the inner boundary of the annulus and, for every black bead that is not already the right-endpoint of a half-circle, add a cross-cut whose exterior endpoint is that black bead.

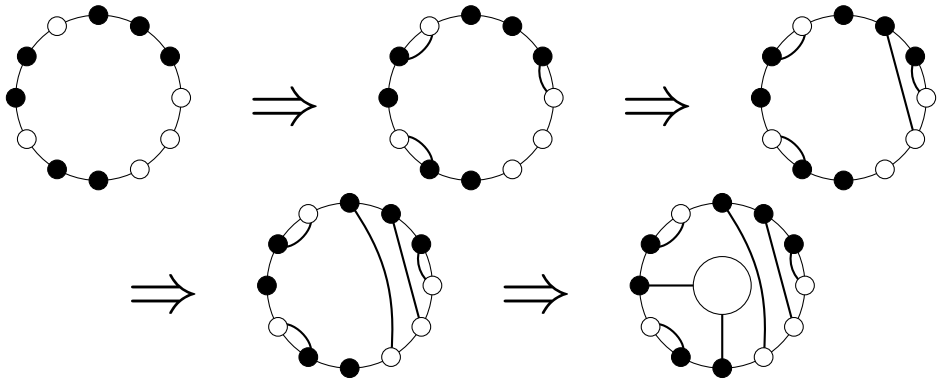
\begin{center}
\begin{tikzpicture}
[scale=1.1,auto=left,every node/.style={circle,inner sep=0pt}]
	\draw (0,0) circle (1cm);
	\node[draw,fill=black,inner sep=2.5pt] (1c) at (60:1) {};
	\node[draw,fill=black,inner sep=2.5pt] (2c) at (30:1) {};
	\node[draw,fill=white,inner sep=2.5pt] (3c) at (0:1) {};
	\node[draw,fill=white,inner sep=2.5pt] (4c) at (330:1) {};
	\node[draw,fill=white,inner sep=2.5pt](5c) at (300:1) {};
	\node[draw,fill=black,inner sep=2.5pt](6c) at (270:1) {};
	\node[draw,fill=black,inner sep=2.5pt] (7c) at (240:1) {};
	\node[draw,fill=white,inner sep=2.5pt] (8c) at (210:1) {};
	\node[draw,fill=black,inner sep=2.5pt] (9c) at (180:1) {};
	\node[draw,fill=black,inner sep=2.5pt] (10c) at (150:1) {};
	\node[draw,fill=white,inner sep=2.5pt] (11c) at (120:1) {};
	\node[draw,fill=black,inner sep=2.5pt] (12c) at (90:1) {};
\end{tikzpicture}
\hspace{0.2in}
\scalebox{2.5}{\raisebox{10pt}{$\Rightarrow$}}
\hspace{0.2in}
\begin{tikzpicture}
[scale=1.1,auto=left,every node/.style={circle,inner sep=0pt}]
	\draw (0,0) circle (1cm);
	\draw[thick, bend left=90] (210:1) to (240:1);
	\draw[thick, bend right=90] (150:1) to (120:1);
	\draw[thick, bend right=40] (30:1) to (0:1);
	\node[draw,fill=black,inner sep=2.5pt] (1c) at (60:1) {};
	\node[draw,fill=black,inner sep=2.5pt] (2c) at (30:1) {};
	\node[draw,fill=white,inner sep=2.5pt] (3c) at (0:1) {};
	\node[draw,fill=white,inner sep=2.5pt] (4c) at (330:1) {};
	\node[draw,fill=white,inner sep=2.5pt](5c) at (300:1) {};
	\node[draw,fill=black,inner sep=2.5pt](6c) at (270:1) {};
	\node[draw,fill=black,inner sep=2.5pt] (7c) at (240:1) {};
	\node[draw,fill=white,inner sep=2.5pt] (8c) at (210:1) {};
	\node[draw,fill=black,inner sep=2.5pt] (9c) at (180:1) {};
	\node[draw,fill=black,inner sep=2.5pt] (10c) at (150:1) {};
	\node[draw,fill=white,inner sep=2.5pt] (11c) at (120:1) {};
	\node[draw,fill=black,inner sep=2.5pt] (12c) at (90:1) {};
\end{tikzpicture}
\hspace{0.2in}
\scalebox{2.5}{\raisebox{10pt}{$\Rightarrow$}}
\hspace{0.2in}
\begin{tikzpicture}
[scale=1.1,auto=left,every node/.style={circle,inner sep=0pt}]
	\draw (0,0) circle (1cm);
	\draw[thick, bend left=90] (210:1) to (240:1);
	\draw[thick, bend right=90] (150:1) to (120:1);
	\draw[thick, bend right=40] (30:1) to (0:1);
	\draw[thick] (60:1) to (330:1);
	\node[draw,fill=black,inner sep=2.5pt] (1c) at (60:1) {};
	\node[draw,fill=black,inner sep=2.5pt] (2c) at (30:1) {};
	\node[draw,fill=white,inner sep=2.5pt] (3c) at (0:1) {};
	\node[draw,fill=white,inner sep=2.5pt] (4c) at (330:1) {};
	\node[draw,fill=white,inner sep=2.5pt](5c) at (300:1) {};
	\node[draw,fill=black,inner sep=2.5pt](6c) at (270:1) {};
	\node[draw,fill=black,inner sep=2.5pt] (7c) at (240:1) {};
	\node[draw,fill=white,inner sep=2.5pt] (8c) at (210:1) {};
	\node[draw,fill=black,inner sep=2.5pt] (9c) at (180:1) {};
	\node[draw,fill=black,inner sep=2.5pt] (10c) at (150:1) {};
	\node[draw,fill=white,inner sep=2.5pt] (11c) at (120:1) {};
	\node[draw,fill=black,inner sep=2.5pt] (12c) at (90:1) {};
\end{tikzpicture}\\
\scalebox{2.5}{\raisebox{10pt}{$\Rightarrow$}}
\hspace{.2in}
\begin{tikzpicture}
[scale=1.1,auto=left,every node/.style={circle,inner sep=0pt}]
	\draw (0,0) circle (1cm);
	\draw[thick, bend left=90] (210:1) to (240:1);
	\draw[thick, bend right=90] (150:1) to (120:1);
	\draw[thick, bend right=20] (300:1) to (90:1);
	\draw[thick, bend right=40] (30:1) to (0:1);
	\draw[thick] (60:1) to (330:1);
	\node[draw,fill=black,inner sep=2.5pt] (1c) at (60:1) {};
	\node[draw,fill=black,inner sep=2.5pt] (2c) at (30:1) {};
	\node[draw,fill=white,inner sep=2.5pt] (3c) at (0:1) {};
	\node[draw,fill=white,inner sep=2.5pt] (4c) at (330:1) {};
	\node[draw,fill=white,inner sep=2.5pt](5c) at (300:1) {};
	\node[draw,fill=black,inner sep=2.5pt](6c) at (270:1) {};
	\node[draw,fill=black,inner sep=2.5pt] (7c) at (240:1) {};
	\node[draw,fill=white,inner sep=2.5pt] (8c) at (210:1) {};
	\node[draw,fill=black,inner sep=2.5pt] (9c) at (180:1) {};
	\node[draw,fill=black,inner sep=2.5pt] (10c) at (150:1) {};
	\node[draw,fill=white,inner sep=2.5pt] (11c) at (120:1) {};
	\node[draw,fill=black,inner sep=2.5pt] (12c) at (90:1) {};
\end{tikzpicture}
\hspace{.2in}
\scalebox{2.5}{\raisebox{10pt}{$\Rightarrow$}}
\hspace{.2in}
\begin{tikzpicture}
[scale=1.1,auto=left,every node/.style={circle,inner sep=0pt}]
	\draw (0,0) circle (.35cm);
	\draw (0,0) circle (1cm);
	\draw[thick, bend left=90] (210:1) to (240:1);
	\draw[thick] (270:1) to (270:.35);
	\draw[thick] (180:1) to (180:.35);
	\draw[thick, bend right=90] (150:1) to (120:1);
	\draw[thick, bend right=20] (300:1) to (90:1);
	\draw[thick, bend right=40] (30:1) to (0:1);
	\draw[thick] (60:1) to (330:1);
	\node[draw,fill=black,inner sep=2.5pt] (1c) at (60:1) {};
	\node[draw,fill=black,inner sep=2.5pt] (2c) at (30:1) {};
	\node[draw,fill=white,inner sep=2.5pt] (3c) at (0:1) {};
	\node[draw,fill=white,inner sep=2.5pt] (4c) at (330:1) {};
	\node[draw,fill=white,inner sep=2.5pt](5c) at (300:1) {};
	\node[draw,fill=black,inner sep=2.5pt](6c) at (270:1) {};
	\node[draw,fill=black,inner sep=2.5pt] (7c) at (240:1) {};
	\node[draw,fill=white,inner sep=2.5pt] (8c) at (210:1) {};
	\node[draw,fill=black,inner sep=2.5pt] (9c) at (180:1) {};
	\node[draw,fill=black,inner sep=2.5pt] (10c) at (150:1) {};
	\node[draw,fill=white,inner sep=2.5pt] (11c) at (120:1) {};
	\node[draw,fill=black,inner sep=2.5pt] (12c) at (90:1) {};
\end{tikzpicture}
\end{center}
Both $\phi_1$ and $\phi_2$ are clearly well-defined and injective, as the excess of $(n+k) - n = k$ black beads are in bijection with the $k$ necessary cross-cuts and the rotational notion of equivalence is identical for combinatorial necklaces and annular non-crossing matchings.  The result then follows.
\end{proof}

\subsection{Enumeration of $Ann_k(2n+k,2m+k)$, General Case}
\label{subsec: enumeration general}

We are finally ready for the enumeration of general $Ann_k(2n+k,2m+k)$ that do not correspond to the special $m=0$ case of Subsection \ref{subsec: m=0 enumeration and necklaces}.  There are actually two sub-cases here depending upon whether or not $k$ is nonzero:

\begin{theorem}
\label{thm: general enumeration}
Let $n$, $m$ , and $k$ be non-negative integers with $m > 0$.  Then:
\begin{enumerate}
\item $\vert Ann_k(2n+k,2m+k) \vert = \vert Ann_0(2n,0) \vert \cdot \vert Ann_0(2m,0) \vert$ \ \ if $k = 0$, and
\item $\vert Ann_k(2n+k,2m+k) \vert = \displaystyle{\frac{k}{(2n+k)(2m+k)} \sum_{d | (2n+k,n,m)} \kern-10pt \phi(d) \binom{(2n+k)/d}{n/d} \binom{(2m+k)/d}{m/d}}$ \ \ if $k > 0$
\end{enumerate}
Where $\phi(d)$ is Euler's totient function and summations run over all common divisors of the given integers.
\end{theorem}
\begin{proof}
Case \#1 follows directly from Proposition \ref{thm: splitting into one-sided product}.  For Case \#2 we look to apply Burnside's Lemma by defining an action of $\Z_{(2n+k)(2m+k)}$ on a set $A$ of relevant matchings.

So fix $n,m \geq 0$, $k >0$, and consider annular non-crossing matchings with precisely $n$ exterior half-circles, $m$ interior half-circles, and $k$ cross-cuts.  To form our set $A$, we require that the $2n+k$ exterior endpoints are located at $\frac{2\pi}{2n+k}$ radian intervals about the exterior boundary, and that the $2m+k$ interior endpoints are located at $\frac{2\pi}{2m+k}$ radian intervals about the interior boundary.  Unlike in the proof of Theorem \ref{thm: maximal cross-cut direct enumeration}, we do not require that the $k$ cross-cuts appear as straight lines that meet the boundaries at right angles, as that condition could require the re-spacing of endpoints on one of the boundary components.  Here we consider cross-cuts up to isotopy that fix their endpoints.  Absolutely no rotational isotopy of endpoints or rotation of either boundary component is allowed.  Notice that $A$ is composed of exactly $k \binom{2n+k}{n} \binom{2m+k}{m}$ matchings.  Here the binomial coefficients are derived from specifying which of the endpoints on the inner and outer boundary component correspond to the left-endpoints of half-circles (as in the proofs of Theorems \ref{thm: maximal cross-cut direct enumeration} and \ref{thm: annular matchings = necklaces}), while the additional $k$ term results from the ambiguity in matching up the remaining $k$ endpoints on each side to form $k$ cross-cuts.  Notice that specifying both endpoints of a single cross-cuts determines how all remaining cross-cuts are matched amongst the remaining $(k-1)$ endpoints on each side.

We then define a left action of $\Z_{(2n+k)(2m+k)}$ on $A$ where $g \cdot a$ is a counter-clockwise rotation of the entire matching by $\frac{2 \pi g}{(2n+k)(2m+k)}$ radians.  If $\vert g \vert = d$, the elements of $A^g$ are matchings that may be radially divided in $d$ identical sub-matchings along both the inner and outer boundaries, with analogous identifications of cross-cuts in each sub-matching.  Observe that we do not require that both endpoints of each cross-cut lie in the same sub-matching, merely that each sub-matching exhibits an identical pattern with regards to any cross-cuts involved.  For $\vert A^g \vert \neq 0$ it is immediate that we must have both $d \mid (2n+k)$ and $d \mid (2m+k)$.  To ensure that left-endpoints of half-circles are mapped to left-endpoints and that cross-cuts are mapped to cross-cuts, it is also required that $d \mid n$, $d \mid m$, and $d \mid k$.  It can easily be shown that $d \mid (2n+k,n,m)$ is necessary and sufficient to satisfy all of these conditions.\footnote{Other equivalent conditions on $d$ include $d \mid (2m+k,n,m)$ and $d \mid (k,n,m)$.}  Thus $\vert A^g \vert \neq 0$ if and only if $d \mid (2n+k,n,m)$.

So take $g \in \Z_{(2n+k)(2m+k)}$ with $\vert g \vert = d$ such that $d \mid (2n+k,n,m)$.  Via similar reasoning as in Theorem \ref{thm: maximal cross-cut direct enumeration}, there are $\binom{2n+k}{n}$ choices for the outer boundary of each sub-matching and $\binom{2m+k}{m}$ independent choices for the inner boundary of each sub-matching.  After the endpoints belonging to cross-cuts have been identified, there are also $k$ independent choices for how the cross-cuts match up across the annulus.  These $k$ choices correspond to similarly-symmetrical matchings whose blocks are identical apart from the fact that their cross-cuts uniformly ``twist around" the annulus by different amounts.  We may conclude that $\vert A^g \vert = k \binom{(2n+k)/d}{n/d} \binom{(2m+k)/d}{m/d}$ if $d \mid (2n+k,n,m)$.

Similarly to our argument from the proof of Theorem \ref{thm: maximal cross-cut direct enumeration}, if $q \vert (2n+k)(2m+k)$ one may show that there are precisely $\phi(q)$ elements $i \in \Z_{(2n+k)(2m+k)}$ with order $\vert i \vert = q$.  It follows that there are precisely $\phi(d)$ elements $g \in \Z_{(2m+k)(2n+k)}$ with $\vert A^g \vert = k \binom{(2n+k)/d}{n/d} \binom{(2m+k)/d}{m/d}$, yielding the equation of Case \#2 via an application of Burnside's Lemma.
\end{proof}

Theorems \ref{thm: maximal cross-cut direct enumeration} and \ref{thm: general enumeration} combine to provide a closed formula for every $\vert Ann_k(2n+k,2m+k) \vert$ in which $n,m,k$ are not all zero.  As we clearly have $\vert Ann_0(0,0) \vert = 1$, this accounts for all possibilities and allows us to directly enumerate $Ann(n,m) = \bigcup_k Ann_k(n,m)$ for all $n,m \geq 0$.

Table \ref{tab: Ann(n,m)} of Appendix \ref{sec: appendix tables} presents values of $\vert Ann(n,m) \vert$ for all $0 \leq n,m \leq 12$, calculated in Maple using the equations of Theorems \ref{thm: maximal cross-cut direct enumeration} and \ref{thm: general enumeration}.  In Proposition \ref{thm: graph interpretation, no crosscuts}, we have already established that the $n=0$ row (or $m=0$ column) of Table \ref{tab: Ann(n,m)} corresponds to A003239 \cite{OEIS}.  Also noted in Section \ref{sec: basic annular results} what the fact that the $n=1$ row (or $m=1$ column) of Table \ref{tab: Ann(n,m)} corresponds to the Catalan numbers.  However, no other rows, diagonals, or triangles of numbers from Table \ref{tab: Ann(n,m)} appear to correspond to any known sequences on OEIS.  This yields an entire family of new integer sequences with an explicit geometric interpretation.  Of particular interest are the rows for $n>2$, representing new generalizations of the Catalan numbers that appear as later terms in a sequence of sequences beginning with A003239 and the Catalan numbers.

Table \ref{tab: Ann(2n)} of Appendix \ref{sec: appendix tables} shows values of $Ann(2n)$ for small values of $n$.  These values are most easily derived by summing anti-diagonals from Table \ref{tab: Ann(n,m)}.  The sequence of Table \ref{tab: Ann(2n)} also fails to appear as a known integer sequence on OEIS.

\section{Acknowledgements}

Both authors would like to thank the Department of Mathematics \& Statistics at Valparaiso University, whose MATH 492 Research in Mathematics course provided the framework under which this research took place.  The first author would also like to thank Dr. Lara Pudwell, who offered helpful advice in the latter stages of the project.

\bibliographystyle{amsplain}
\bibliography{biblio}

\newpage

\appendix
\section{Tables of Values}
\label{sec: appendix tables}

All tables in this appendix were generated in Maple 18 using the equations of Theorems \ref{thm: maximal cross-cut direct enumeration} and \ref{thm: general enumeration}.  Coding is available upon request from Paul Drube \tt (paul.drube@valpo.edu)

\begin{table}[h!]
\centering
\caption{Enumeration of Maximal Cross-Cut Annular Matchings $\vert Ann_k(2n+k,k) \vert$.}
\label{tab: Annk(2n+k,k)}
\begin{small}
\begin{tabular}{c|c c c c c c c c c c c }
& k=0 & k=1 & k=2 & k=3 & k=4 & k=5 & k=6 & k=7 & k=8 & k=9 & k=10 \\ \hline
n=0 & 1 & 1 & 1 & 1 & 1 & 1 & 1 & 1 & 1 & 1 & 1 \\
n=1 & 1 & 1 & 1 & 1 & 1 & 1 & 1 & 1 & 1 & 1 & 1 \\
n=2 & 2 & 2 & 3 & 3 & 4 & 4 & 5 & 5 & 6 & 6 & 7 \\
n=3 & 4 & 5 & 7 & 10 & 12 & 15 & 19 & 22 & 26 & 31 & 35 \\
n=4 & 10 & 14 & 22 & 30 & 43 & 55 & 73 & 91 & 116 & 140 & 172 \\
n=5 & 26 & 42 & 66 & 99 & 143 & 201 & 273 & 364 & 476 & 612 & 776 \\
n=6 & 80 & 132 & 217 & 335 & 504 & 728 & 1038 & 1428 & 1944 & 2586 & 3399 \\
n=7 & 246 & 429 & 715 & 1144 & 1768 & 2652 & 3876 & 5538 & 7752 & 10659 & 14421 \\
n=8 & 810 & 1430 & 2438 & 3978 & 6310 & 9690 & 14550 & 21318 & 30667 & 43263 & 60115 \\
n=9 & 2704 & 4862 & 8398 & 14000 & 22610 & 35530 & 54484 & 81719 & 120175 & 173593 & 246675 \\
n=10 & 9252 & 16796 & 29414 & 49742 & 81752 & 130752 & 204347 & 312455 & 468754 & 690690 & 1001603
\end{tabular}
\end{small}
\end{table}

\begin{table}[h!]
\centering
\caption{Enumeration of Annular Non-Crossing Matchings $\vert Ann(n,m) \vert$.  Entries where $\vert Ann(n,m) \vert =0$ have been left blank, and the main diagonal $\vert Ann(n,n) \vert$ has been bolded to emphasize $n \leftrightarrow m$ symmetry.}
\label{tab: Ann(n,m)}
\begin{small}
\begin{tabular}{c|c c c c c c c c c c c c c}
& m=0 & m=1 & m=2 & m=3 & m=4 & m=5 & m=6 & m=7 & m=8 & m=9 & m=10 & m=11 & m=12 \\ \hline
n=0 & \textbf{1} & & 1 & & 2 & & 4 & & 10 & & 26 & & 80 \\
n=1 & & \textbf{1} & & 1 & & 2 & & 5 & & 14 & & 42 & \\
n=2 & 1 & & \textbf{2} & & 3 & & 7 & & 17 & & 48 & & 146 \\
n=3 & & 1 & & \textbf{2} & & 3 & & 8 & & 24 & & 72 & \\
n=4 & 2 & & 3 & & \textbf{7} & & 14 & & 38 & & 106 & & 335 \\
n=5 & & 2 & & 3 & & \textbf{8} & & 20 & & 60 & & 189 & \\
n=6 & 4 & & 7 & & 14 & & \textbf{34} & & 90 & & 263 & & 834 \\
n=7 & & 5 & & 8 & & 20 & & \textbf{58} & & 175 & & 560 & \\
n=8 & 10 & & 17 & & 38 & & 90 & & \textbf{255}  & & 750 & & 2420 \\
n=9 & & 14 & & 24 & & 60 & & 175 & & \textbf{546} & & 1764 & \\
n=10 & 26 & & 48 & & 106 & & 263 & & 750 & & \textbf{2268} & & 7372 \\
n=11 & & 42 & & 72 & & 189 & & 560 & & 1764 & & \textbf{5774} & \\
n=12 & 80 & & 146 & & 335 & & 834 & & 2420 & & 7372 & & \textbf{24198}
\end{tabular}
\end{small}
\end{table}

\begin{table}[h!]
\centering
\caption{Enumeration of Annular Non-Crossing Matchings $\vert Ann(2n) \vert$.}
\label{tab: Ann(2n)}
\begin{tabular}{c|c c c c c c c c c c c c c c}
n & 0 & 1 & 2 & 3 & 4 & 5 & 6 & 7 & 8 & 9 & 10 & 11 & 12 & 13\\ \hline
Ann(2n) & 1 & 3 & 8 & 20 & 57 & 166 & 538 & 1762 & 6045 & 21040 & 74628 & 267598 & 970134 & 3544416
\end{tabular}
\end{table}

\end{document}